\documentclass[12pt,oneside,a4paper,reqno]{amsart}

\usepackage{amssymb}
\usepackage{amsmath}
\usepackage{amsthm}
\usepackage{amscd}
\usepackage[all]{xy}
\usepackage{longtable}
\usepackage{mathrsfs}

\usepackage{comment}
\usepackage{todonotes}

\usepackage{xcolor}
\usepackage{pict2e}
\usepackage{graphicx}
\usepackage{hyperref}
\hypersetup{
	colorlinks=true,
	linkcolor=red,
	citecolor=blue}
	\usepackage[utf8]{inputenc}
\usepackage[T1]{fontenc}
\usepackage{color}

\usepackage{extarrows}
\usepackage{tikz}
\usetikzlibrary{cd}

\usepackage{geometry}
\theoremstyle{plain}
\newtheorem{theorem}{Theorem}[section]
\newtheorem{lemma}[theorem]{Lemma}
\newtheorem{corollary}[theorem]{Corollary}
\newtheorem{proposition}[theorem]{Proposition}
\newtheorem{remark}[theorem]{Remark}

\newtheorem{problem}[theorem]{Problem}

\theoremstyle{definition}

\newtheorem{conjecture}[theorem]{Conjecture}

\newtheorem{remark-theorem}[theorem]{Remark-Theorem}
\newtheorem{counterexample}[theorem]{Counterexample}

\newtheorem{setup}[theorem]{Setup}

\newcommand{\kah}{K\"{a}hler }
\newcommand{\idd}{i\partial\overline{\partial}}
\newcommand{\dbar}{\overline{\partial}}

\newcommand{\cal}[1]{\mathcal{#1}}
\newcommand{\bb}[1]{\mathbb{#1}}
\newcommand{\scr}[1]{\mathscr{#1}}
\newcommand{\rom}[1]{\mathrm{#1}}

\newcommand{\Ox}[1]{\cal{O}_X({#1})}
\newcommand{\tl}[1]{\widetilde{#1}}

\newcommand{\hldpn}{\hbar^{L,\psi}_{D,\nu}}
\newcommand{\hldnu}{\hbar^{L}_{D,\nu}}
\newcommand{\hdt}{\hbar^{\psi,\tau}_{D,\nu}}
\newcommand{\opnd}{\omega_{P,\nu,\delta}}

\newcommand{\Pt}{Poincar\'{e}-type }

\newcommand{\lara}[2]{\langle{#1},{#2}\rangle}

\newcommand{\iO}[1]{i\Theta_{#1}}


\begin{document}
\title
[Hard Lefschetz Theorem for Logarithmic Sheaves] 
{Hard Lefschetz Theorem for Logarithmic Sheaves Twisted by Pseudo-effective Line Bundles}
\author{Yuta Watanabe}
\address{Department of Mathematics, Faculty of Science and Engineering, Chuo University.
1-13-27 Kasuga, Bunkyo-ku, Tokyo 112-8551, Japan}
\email{{\tt wyuta.math@gmail.com}, {\tt wyuta@math.chuo-u.ac.jp}}

\begin{abstract}
    In this paper, a Hard Lefschetz theorem for logarithmic sheaves twisted by pseudo-effective line bundles is established, extending a theorem of Demailly, Peternell, and Schneider.    
\end{abstract}


\maketitle

\vspace{-5mm}




\section{Introduction}

In complex analytic geometry, pseudo-effective line bundles have become a central object of study in modern research. 
From a differential-geometric perspective, they are characterized by the existence of singular Hermitian metrics whose curvature currents are semi-positive in the sense of currents. 
Multiplier ideal sheaves have become an essential tool for analyzing the singularities of such metrics and play an important role both in pluripotential theory and in applications to algebraic geometry. 
In this context, the following Hard Lefschetz theorem for pseudo-effective line bundles is a fundamental result in complex geometry and has applications to the study of cohomology groups, including vanishing and non-vanishing results, as well as injectivity theorems.

\begin{theorem}[{\cite[Theorem 0.1]{DPS01}}]\label{Theorem: Theorem 0.1 in DPS01}
    Let $X$ be a compact \kah manifold of dimension $n$ with a \kah metric $\omega$ and $L\longrightarrow X$ be a pseudo-effective line bundle, that is, $L$ admits a singular Hermitian metric $h$ with a semi-positive curvature current on $X$. 
    Then, the wedge multiplication operator $\omega^q\wedge\bullet$ induces a surjective morphism 
    \begin{align*}
        \Phi^q_{\omega,h}:H^0(X,\Omega_X^{n-q}\otimes L\otimes\scr{I}(h))\longrightarrow H^q(X,\Omega_X^n\otimes L\otimes \scr{I}(h))
    \end{align*}
    for every nonnegative integer $q$, where $\scr{I}(h)$ is the associated multiplier ideal sheaf. 
\end{theorem}

The special case when $L$ is nef is due to Takegoshi \cite{Tak97} (for the definition of nef in the analytic setting, cf. \cite{Dem12}). 
In particular, the multiplier ideal sheaf \(\scr{I}(h)\) in Theorem \ref{Theorem: Theorem 0.1 in DPS01} is not merely a technical assumption but is essential. 
Indeed, when \(L\) is nef, there are examples in which a multiplier ideal sheaf \(\scr{I}(h)\neq\cal{O}_X\) plays an essential role; in such cases, the map \(\Phi^q_{\omega,h}\) is not surjective if the multiplier ideal sheaf is not taken into account and is replaced by $\cal{O}_X$ (see \cite[$\S$2.5 or Proposition 2.5.1]{DPS01}).
In the further special case where $L$ is semi-positive, the multiplier ideal sheaf $\scr{I}(h)$ coincides with $\cal{O}_X$, and we obtain the result of \cite[Corollary 2.1.2]{DPS01}, which was already observed by Enoki \cite{Eno93} and Mourougane \cite{Mou95}. 
Subsequently, results on closedness and harmonicity that complement Theorem \ref{Theorem: Theorem 0.1 in DPS01} were obtained by Wu \cite{Wu21}.

In this paper, using the recently established analytic logarithmic \(L^2\)-Dolbeault resolution (see \cite{HLWY23,Wat26a}), we extend Theorem \ref{Theorem: Theorem 0.1 in DPS01} to the setting of logarithmic sheaves.
To obtain this logarithmic \(L^2\)-Dolbeault resolution, a specific choice of singular Hermitian metric is required, which makes it difficult to take \(L^2\)-weak limits along an approximation process while controlling the resolution. 
We overcome this difficulty by applying Demailly's convolution-type regularization \cite{Dem82} via the exponential map \(\exp:T_X\longrightarrow X\) using a kernel symmetric with respect to a \kah metric, thereby constructing a sequence of singular Hermitian metrics that increases to \(h\) and preserves the resolution appropriately.

\begin{theorem}\label{Theorem: Hard Lefschetz if TX>0 on D}
    Let $X$ be a compact \kah manifold of dimension $n$ and $D=\sum^J_{j=1}D_j$ be a simple normal crossing divisor. 
    Let $\omega_P$ be a smooth \kah metric on $X\setminus D$ which is of Poincar\'{e}-type along $D$.
    Let $L\longrightarrow X$ be a pseudo-effective line bundle, that is, $L$ admits a singular Hermitian metric $h$ with semi-positive curvature current on $X$. 
    If the tangent bundle $T_X$ is Griffiths semi-positive on $D$, then the wedge multiplication operator $\omega_P^q\wedge\bullet$ induces a surjective morphism 
    \begin{align*}
        \Phi^q_{\omega_P,h}:H^0(X,\Omega_X^{n-q}(\log D)\otimes L\otimes\scr{I}(h))\longrightarrow H^q(X,\Omega_X^n(\log D)\otimes L\otimes \scr{I}(h))
    \end{align*}
    for every nonnegative integer $q$.
\end{theorem}

Here, the map $\Phi^q_{\omega_P,h}$ is well-defined (see Proposition \ref{Proposition: well-definedness of Phi}).
As an analogue of Theorem \ref{Theorem: Hard Lefschetz if TX>0 on D}, the same surjectivity result for \(\Phi^q_{\omega_P,h}\) holds if each pseudo-effective line bundle \(\Ox{D_j}\) admits a minimal singular Hermitian metric whose positive Lelong number locus is empty (see Corollary \ref{Corollary: Hard Lefschetz for h_min}).
We then consider the relationship with the map \(\times \sigma_D\) induced by multiplication by the defining section \(\sigma_D\) of \(D\).
\[
\begin{tikzcd}[
  row sep=normal,
  column sep=tiny
]
\hspace{-4mm}H^0(X,\Omega_X^{n-q}\!\otimes\!\cal{O}_{\!X\!}(D)\!\otimes\! L\otimes\!\scr{I}(h))\!\! \arrow[rrrrrrrrrr, "\omega^q\wedge\,\bullet", "\Phi^q_{\omega,h,D}"'] & & & & & & & & & & \!H^q(X,\Omega_X^n(\log D)\!\otimes\! L\otimes\!\scr{I}(h)) \\
& \hspace{-28mm}H^0(X,\Omega_X^{n-q}(\log D)\!\otimes\! L\otimes\!\scr{I}(h))\hspace{-10mm} \arrow[ul, hookrightarrow] \arrow[urrrrrrrrr, "\omega^q_P\wedge\,\bullet"'] & & & & & & & & & &\\
\hspace{-6mm}H^0(X,\Omega_X^{n-q}\!\otimes\! L\otimes\!\scr{I}(h)) \arrow[rrrrrrrrrr, twoheadrightarrow, "\omega^q\wedge\,\bullet"] \arrow[uu, hookrightarrow, "\times\sigma_D"] \arrow[ur, hookrightarrow] & & & & & & & & & & H^q(X,\Omega_X^n\!\otimes\! L\otimes\!\scr{I}(h)). \arrow[uu, "\times\sigma_D\,:=\,\Phi_{\sigma_D}"']
\end{tikzcd}
\]
In Section \ref{Section: Proofs}, based on the proof and observations concerning Theorem \ref{Theorem: Hard Lefschetz if TX>0 on D}, we propose the conjecture that, without the assumption on \(D\), the map \(\Phi^q_{\omega_P,h}\) is surjective when its target is restricted to \(\rom{Im}\,\Phi_{\sigma_D}=\rom{Im}(\times \sigma_D)\) (see Conjecture \ref{Conjecture: Conjecture without assumption of D}). 

Note that $\Omega^n_X(\log D)\simeq K_X\otimes\Ox{D}$. 
If \(\Ox{D}\) is semi-positive, that is, it admits a smooth Hermitian metric $h_D$ with semi-positive curvature, then \(\Phi^q_{\omega,h,D}\) is surjective by Theorem \ref{Theorem: Theorem 0.1 in DPS01}.
Furthermore, if it is additionally assumed that $\iO{L,h}\geq\varepsilon\iO{\cal{O}_{\!X\!}(D),h_D}$ in the sense of currents for some $\varepsilon>0$, then $\Phi_{\sigma_D}$ is injective by the Enoki-type injectivity theorem (see \cite{FM21}).
Taking these facts into account, by imposing an appropriate semi-positivity assumption on \(\Ox{D}\), we obtain the surjectivity of \(\Phi^q_{\omega_P,h}\) as follows.

\begin{theorem}\label{Theorem: Hard Lefschetz for semi-positivity}
    Let $X,D,\omega_P,L,h$ be as in the setting of Theorem \ref{Theorem: Hard Lefschetz if TX>0 on D}.
    If each line bundle $\cal{O}_X(D_j)$ is semi-positive, then the wedge multiplication operator $\omega_P^q\wedge\bullet$ induces a surjective morphism 
    \begin{align*}
        \Phi^q_{\omega_P,h}:H^0(X,\Omega_X^{n-q}(\log D)\otimes L\otimes\scr{I}(h))\longrightarrow H^q(X,\Omega_X^n(\log D)\otimes L\otimes \scr{I}(h))
    \end{align*}
    for every nonnegative integer $q$.
\end{theorem}

Finally, we show that the multiplier ideal sheaf plays an essential role even in the setting of logarithmic sheaves (see Remark \ref{Remark: counterexample for nef case}). 
To this end, following \cite{DPS01}, an example (see Counterexample \ref{Counterexample: counterexample following DPS01}) is given in which the surjectivity fails when $L$ and $D$ are nef and the multiplier ideal sheaf is not taken into account and is replaced by $\cal{O}_X$.
Here, the assumptions on \(D\) in Theorems \ref{Theorem: Hard Lefschetz if TX>0 on D} and \ref{Theorem: Hard Lefschetz for semi-positivity} are stronger than nefness (see Remark \ref{Remark: TX Grif>0 and nefness}), and thus the following problem naturally arises.

\begin{problem}
    Does Theorem \ref{Theorem: Hard Lefschetz if TX>0 on D} or \ref{Theorem: Hard Lefschetz for semi-positivity} still hold if \(D\) is simply assumed to be nef, or is there a counterexample?
\end{problem}

\section{Preliminaries}

\subsection{A Bochner-type formula}

Let $(L,h)$ be a smooth Hermitian line bundle on a (non necessarily compact) \kah manifold $(Y,\omega)$. 
We denote by $|\bullet|_{h,\omega}=|\bullet|_{\wedge^{p,q}\omega\otimes h}$ the pointwise Hermitian norm on $\bigwedge^{p,q}T^*_Y\otimes L$ associated with $\omega$ and $h$, and by $||\bullet||_{h,\omega}$ the global $L^2$-norm 
$\displaystyle||u||^2_{h,\omega}=\int_Y|u|^2_{h,\omega}dV_\omega$, where $\displaystyle dV_\omega=\frac{\omega^n}{n!}$. 
We consider the $\dbar$-operator acting on $(p,q)$-forms with values in $L$, its adjoint $\dbar^*_{h,\omega}$ with respect to $h$ and $\omega$.  

According to \cite{DPS01}, for any smooth $L$-valued $(n-q,0)$-form $v$ with compact support, we have 
\begin{align*}
    \big(\dbar^*_{h,\omega}\dbar+\dbar\dbar^*_{h,\omega}\big)(\omega^q\wedge v)-\omega^q\wedge\big(\dbar^*_{h,\omega}\dbar v\big)=q\,\iO{L,h}\wedge\omega^{q-1}\wedge v.
\end{align*}
Combining this with $\lara{\omega^q\wedge\bullet}{\omega^q\wedge\bullet}_\omega=(q!)^2\lara{\bullet}{\bullet}_\omega$ for $(n-q,0)$-forms, which will be shown later, and $\Lambda_\omega(\omega^q\wedge v)=q(q+1)\omega^{q-1}\wedge v$, we obtain the following key equality.

\begin{proposition}[{\cite[Proposition 2.3.3]{DPS01}}]\label{Proposition: Prop in DPS01}
    Let $(Y,\omega)$ be a complete \kah manifold of dimension $n$ and $(L,h)$ be a smooth Hermitian line bundle such that the curvature possesses a uniform lower bounded $\iO{L,h}\geq-C\omega$. 
    For every measurable $(n-q,0)$-form $v$ with $L^2$-coefficients and values in $L$ such that $u=\omega^q\wedge v$ has differentials $\dbar u, \dbar^*_{h,\omega}u$ also in $L^2$, we have 
    \begin{align*}
        \big|\big|\dbar u\big|\big|^2_{h,\omega}+\big|\big|\dbar^*_{h,\omega}u\big|\big|^2_{h,\omega}=(q!)^2\big|\big|\dbar v\big|\big|^2_{h,\omega}+\frac{1}{q+1}\int_Y\sum_{|J|=q}\Bigl(\sum_{j\in J}\lambda_j\Bigr)|u_J|^2_{h,\omega}dV_\omega,
    \end{align*}
    where $\lambda_1,\ldots,\lambda_n$ are the curvature eigenvalues of $\iO{L,h}$ with respect to $\omega$.
\end{proposition}

Lemma \ref{Lemma: inequality of (n,0)-forms} is obtained from straightforward computations. 

\begin{lemma}\label{Lemma: inequality of (n,0)-forms}
    Let $\gamma_1$ and $\gamma_2$ be Hermitian metrics on $X$ with $\gamma_1\geq\gamma_2$. 
    Then we have 
    \begin{itemize}
        \item the equality $|u|^2_{\gamma_1} dV_{\gamma_1}=|u|^2_{\gamma_2} dV_{\gamma_2}$ holds for any $(n,0)$-form $u$, 
        \item the inequality $|u|^2_{\gamma_1} dV_{\gamma_1}\leq|u|^2_{\gamma_2} dV_{\gamma_2}$ holds for any $(n,q)$-form $u$ and any $q\geq1$,
        \item the inequality $|u|^2_{\gamma_1} dV_{\gamma_1}\geq|u|^2_{\gamma_2} dV_{\gamma_2}$ holds for any $(p,0)$-form $u$ and any $p\geq1$.
    \end{itemize}
\end{lemma}

\subsection{Logarithmic $L^2$-Dolbeault resolusions and the map $\Phi^q_{h,\omega_P}$}

In this subsection, we introduce various (logarithmic) $L^2$-Dolbeault resolution. 
The $L^2$-Dolbeault resolusion involving multiplier ideal sheaves was first established for $(n,q)$-forms in order to prove Nadel vanishing and was later generalized to $(p,q)$-forms as follows.

Let $X$ be a complex manifold and $L\longrightarrow X$ be a holomorphic line bundle with a singular Hermitian metric $h$. 
We define the subsheaf $\scr{L}^{p,q}_{L,h}$ of germs of $(p,q)$-forms $u$ with values in $L$ and with measurable coefficients such that both $|u|^2_h$ and $\big|\dbar u\big|^2_h$ are locally integrable, here we see that $\scr{L}^{p,q}_{L,h}$ is a fine sheaf.

\begin{theorem}[{\cite[Theorem 5.3]{Wat25}, cf. \cite{Dem12}}]\label{Theorem: L2-Dolbeault resolusion}
    Let $L$ be a holomorphic line bundle on a complex manifold $X$ with a singular Hermitian metric $h$. 
    If the local weights of $h$ are quasi-plurisubharmonic, then for any $p\geq0$, the $L^2$-Dolbeault complex 
    \begin{align*}
        0\longrightarrow \Omega_X^p\otimes\cal{O}_X(L)\otimes\scr{I}(h)\hookrightarrow\scr{L}^{p,0}_{L,h}\overset{\dbar}{\longrightarrow}\scr{L}^{p,1}_{L,h}\overset{\dbar}{\longrightarrow}\scr{L}^{p,2}_{L,h}\overset{\dbar}{\longrightarrow}\cdots
    \end{align*}
    is exact on $X$. Thus, for any $p,q\geq0$, we have the $L^2$-Dolbeault isomorphism 
    \begin{align*}
        H^q(X,\Omega^p_X\otimes L\otimes\scr{I}(h))\cong H^q\big(\Gamma(X,\scr{L}^{p,\ast}_{L,h})\big).
    \end{align*}
\end{theorem}

The analytic logarithmic \(L^2\)-Dolbeault resolusion for the logarithmic sheaf \(\Omega^p(\log D)\) was first established in \cite{HLWY23} for a smooth Hermitian metric on a line bundle. 
It was subsequently generalized in \cite{Wat26a} to singular Hermitian metrics in the form twisted by multiplier ideal sheaves as follows.
Let $D$ be a simple normal crossing divisor and $\omega_P$ be a smooth \kah metric on $X\setminus D$. 
Let $h^L_D$ be a singular Hermitian metric on $L|_{X\setminus D}$. The sheaf $\scr{L}^{p,q}_{L,h^L_D,\,\omega_P}$ over $X$ is defined by the following: 
for any open subset $U$ of $X$, the section space $\scr{L}^{p,q}_{L,h^L_D,\,\omega_P}(U)$ consists of $L$-valued $(p,q)$-forms $u$ with measurable coefficients such that $|u|^2_{h^L_D,\,\omega_P}dV_{\omega_P}$ and $\big|\dbar u\big|^2_{h^L_D,\,\omega_P}dV_{\omega_P}$ are integrable on $U\setminus D$.

\begin{theorem}[{\cite[Theorem 3.7]{Wat26a}}]\label{Theorem: Logarithmic L2-Dolbeault resolusion}
    Let $X$ be a compact \kah manifold, $D=\sum^J_{j=1} D_j$ be a simple normal crossing divisor on $X$, and $\omega_P$ be a smooth \kah metric on $X\setminus D$ which is of Poincar\'{e}-type along $D$. 
    Let $L$ be a holomorphic line bundle on $X$ with a singular Hermitian metric $h$ whose local weights are quasi-plurisubharmonic, and $\sigma_j$ be the defining section of $D_j$. 
    Fix smooth Hermitian metrics $h_j$ on $\cal{O}_X(D_j)$.  
    
    Then, there exists a sufficiently large integer $\alpha>0$ depending on the compactness of $X$ such that the singular Hermitian metric $\hbar^L_D$ on $L|_{X\setminus D}$ defined by 
    \begin{align*}
        \hbar^L_D=h\prod^J_{j=1}|\sigma_j|^2_{h_j}(\log |\sigma_j|^2_{h_j})^{2\alpha}
    \end{align*}
    has the property that, for any $p\geq0$, the logarithmic $L^2$-Dolbeault complex 
    \begin{align*}
        0\longrightarrow \Omega_X^p(\log D)\otimes\cal{O}_X(L)\otimes\scr{I}(h)\hookrightarrow\scr{L}^{p,0}_{L,\hbar^L_D,\,\omega_P}\overset{\dbar}{\longrightarrow}\scr{L}^{p,1}_{L,\hbar^L_D,\,\omega_P}\overset{\dbar}{\longrightarrow}\scr{L}^{p,2}_{L,\hbar^L_D,\,\omega_P}\overset{\dbar}{\longrightarrow}\cdots
    \end{align*}
    is exact on $X$; that is, the complex $\Big(\scr{L}^{p,\ast}_{L,\hbar^L_D,\,\omega_P},\dbar\Big)$ is logarithmic $L^2$-Dolbeault resolution of $\Omega_X^p(\log D)\otimes\cal{O}_X(L)\otimes\scr{I}(h)$. 
    In particular, the complex is exact for $q=0$, that is, 
    \begin{align*}
        \Omega_X^p(\log D)\otimes\cal{O}_X(L)\otimes\scr{I}(h)=\rom{Ker}\Bigl(\,\dbar:\scr{L}^{p,0}_{L,\hbar^L_D,\,\omega_P}\longrightarrow \scr{L}^{p,1}_{L,\hbar^L_D,\,\omega_P}\Bigr).
    \end{align*}

    Thus, we have the logarithmic $L^2$-Dolbeault isomorphism 
    \begin{align*}
        H^q(X,\Omega_X^p(\log D)\otimes L\otimes\scr{I}(h))\cong H^q\Bigl(\Gamma\Bigl(X\setminus D,\mathscr{L}^{p,\ast}_{L,\hbar^L_D,\,\omega_P}\Bigr)\Bigr)
    \end{align*}
    for any nonnegative integers $p,q\geq0$.
\end{theorem}

Furthermore, these results extend appropriately to the setting of vector bundles when the singular Hermitian metric has a suitable positivity property such as Griffiths or Nakano positivity (see \cite{Wat25,Wat26a}).
Finally, using the logarithmic \(L^2\)-Dolbeault resolution, we establish the well-definedness of the Hard Lefschetz-type map induced by the wedge multiplication operator associated with a \Pt \kah metric.

\begin{proposition}\label{Proposition: well-definedness of Phi}
    Let $X,D,L,h$ be as in the setting of Theorem \ref{Theorem: Logarithmic L2-Dolbeault resolusion}. For any smooth \kah metric $\omega_P$ on $X\setminus D$ which is of \Pt along $D$, the linear map 
    \begin{align*}
        \Phi^q_{\omega_P,h}:H^0(X,\Omega_X^{n-q}(\log D)\otimes L\otimes\scr{I}(h))\longrightarrow H^q(X,\Omega_X^n(\log D)\otimes L\otimes \scr{I}(h))
    \end{align*}
    given by $u\mapsto\{\omega^q_P\wedge u\}$ is well-defined. 
\end{proposition}

\begin{proof}
    Let $\omega$ be a \kah metric on $X$ and $v$ be an $(n-q,0)$-form on $X$. 
    For a fixed point $x_0\in X$, there exists an orthonormal basis $(\partial/\partial z_1,\ldots,\partial/\partial z_n)$ of $T_{X,x_0}$ such that 
    \begin{align*}
        \omega_{x_0}=i\sum_{1\leq j\leq n}dz_j\wedge d\overline{z}_j, \qquad v(x_0)=\sum_{|I|=n-q}v_Idz_I.
    \end{align*}
    A straightforward calculation gives $\omega^q_{x_0}=q!\sqrt{-1}^{q^2}\sum_{|J|=q}dz_J\wedge d\overline{z}_J$, and hence 
    \begin{align*}
        \omega^q\wedge v(x_0)&=q!\sqrt{-1}^{q^2}\!\!\!\!\!\!\!\!\sum_{|J|=q,|I|=n-q}\!\!\!\!\!v_Idz_J\wedge d\overline{z}_J\wedge dz_I=q!\sqrt{-1}^{q^2}\!\!\sum_{|I|=n-q}v_Idz_{I^c}\wedge d\overline{z}_{I^c}\wedge dz_I\\
        &=q!\sqrt{-1}^{q^2}\!\!\sum_{|I|=n-q}\pm\,v_Idz_1\wedge\cdots\wedge dz_n\wedge d\overline{z}_{I^c}.
    \end{align*}
    Therefore, we obtain $|\omega^q\wedge v(x_0)|^2_\omega=(q!)^2\sum_{|I|=n-q}|v_I|^2=(q!)^2|v(x_0)|^2_\omega$,
    Since $x_0$ was arbitrary, this equality holds on $X$. Consequently, in particular, we have 
    \begin{align*}
        |\omega^q_P\wedge v|^2_{h,\omega_P}=(q!)^2|v|^2_{h,\omega_P}
    \end{align*}
    on $X\setminus D$ for any $L$-valued $(n-q,0)$-form $v$. 

    Given a singular Hermitian metric \(h\), let \(h^L_D\) be the singular Hermitian metric on \(L|_{X\setminus D}\) constructed as in Theorem \ref{Theorem: Logarithmic L2-Dolbeault resolusion} using smooth Hermitian metrics and defining sections associated with each \(D_j\). 
    Then, using the above equality, for any global section \(u\in H^0(X,\Omega^q_X(\log D)\otimes L\otimes\scr{I}(h))=\Gamma\Bigl(X\setminus D,\scr{L}^{n-q,0}_{L,h^L_D,\omega_P}\Bigr)\cap\rom{Ker}\,\dbar\), we have
    \begin{align*}
        \int_{X\setminus D}|\omega^q_P\wedge u|^2_{h^L_D,\omega_P}dV_{\omega_P}=(q!)^2\int_{X\setminus D}|u|^2_{h^L_D,\omega_P}dV_{\omega_P}<+\infty.
    \end{align*}
    Combining this integrability with $\dbar(\omega^q_P\wedge u)=\omega^q_P\wedge\dbar u=0$ on $X\setminus D$, it follows that $\omega^q_P\wedge u\in \Gamma\Bigl(X\setminus D,\scr{L}^{n,q}_{L,h^L_D,\omega_P}\Bigr)$. 
    In particular, we obtain 
    \begin{align*}
        \{\omega^q_P\wedge u\}\in H^q\Bigl(\Gamma\Bigl(X\setminus D,\scr{L}^{n,\ast}_{L,h^L_D,\omega_P}\Bigr)\Bigr)\cong H^q(X,\Omega^n_X(\log D)\otimes L\otimes\scr{I}(h)).
    \end{align*}
    Hence, the map $\Phi^q_{\omega_P,h}$ is well-defined.
\end{proof}

\subsection{Demailly's approximations}

We know the following Demailly's approximation as an extremely effective method.

\begin{theorem}[{\cite[Theorem\,2.2.1]{DPS01}}]\label{Theorem: Demailly approximation preserves ideal sheaves}
    Let $X$ be a compact complex manifold with a Hermitian metric $\omega$ and $T=\alpha+\idd\varphi$ be a closed $(1,1)$-current on $X$ 
    where $\alpha$ is a smooth closed $(1,1)$-form and $\varphi$ is a quasi-plurisubharmonic function. 
    If $T=\alpha+\idd\varphi\geq \gamma$ holds for a continuous real $(1,1)$-from $\gamma$ on $X$, then there exists a sequence of quasi-plurisubharmonic functions $\{\varphi_\nu\}_{\nu\in\bb{N}}$ on $X$ with the following properties:
    \begin{itemize}
        \item [($a$)] $\varphi_\nu$ is smooth on $X\setminus Z_\nu$, where $Z_\nu$ is an analytic subset of $X$ with $Z_\nu\subset Z_{\nu+1}$,
        \item [($b$)] $\{\varphi_\nu\}_{\nu\in\bb{N}}$ is a decreasing sequence of functions converging to $\varphi$,
        \item [($c$)] $\scr{I}(\varphi)=\scr{I}(\varphi_\nu)$ on $X$ for all $\nu\in\bb{N}$,
        \item [($d$)] $T_\nu:=\alpha+\idd\varphi_\nu$ satisfies $T_\nu\geq\gamma-\varepsilon_\nu\omega$, where $\{\varepsilon_\nu\}_{\nu\in\bb{N}}$ is a decreasing sequence of nonnegative real numbers converging to $0$. 
    \end{itemize}
\end{theorem}

Unlike the so-called Demailly approximation of quasi-plurisubharmonic functions described above, which uses the Ohsawa-Takegoshi extension theorem and Bergman kernels, we introduce the convolution-type regularization via the exponential map $\exp:T_X\longrightarrow X$ using a kernel symmetric with respect to a \kah metric, which was given earlier by Demailly.
Since this regularization approximates \(h\) by smooth functions, its drawback is that information on the singularities of \(h\), in particular on the associated multiplier ideal sheaf $\scr{I}(h)$, is lost. However, this approximation is effective in the present setting because the Hermitian metric on each $\Ox{D_j}$ in the logarithmic \(L^2\)-Dolbeault resolution is required to be smooth.

\begin{theorem}[{\cite[Th\'{e}or\`{e}m 9.1]{Dem82}}]\label{Theorem: Demailly approximation of smooth functions}
    Let $X$ be a complex manifold equipped with a \kah metric $\omega$ and $\varphi$ be a quasi-plurisubharmonic function on $X$. 
    Assume that there exists a continuous real $(1,1)$-form $\theta$ such that $\idd\varphi\geq\theta$ in the sense of currents. 
    Then, there exist a decreasing family $\{\varphi_\nu\}_{\nu\in\bb{N}}$ of smooth functions on $X$, a family $\{\gamma_\nu\}_{\nu\in\bb{N}}$ of continuous real $(1,1)$-forms, and a decreasing family $\{\lambda_\nu\}_{\nu\in\bb{N}}$ of nonnegative continuous functions on $X$ satisfying the following properties: 
    \begin{itemize}
        \item [$(a)$] $\displaystyle\lim_{\nu\to+\infty}\varphi_\nu(x)=\varphi(x)$ for every $x\in X$, 
        \item [$(b)$] $\idd\varphi_\nu\geq\gamma_\nu-\lambda_\nu\omega$ and $\gamma_\nu\geq\theta$, 
        \item [$(c)$] $\gamma_\nu\longrightarrow(\idd\varphi)_{ac}$ almost everywhere on $X$ as $\nu\to+\infty$,
        \item [$(d)$] $\lambda_\nu\longrightarrow0$ almost everywhere on $X$, more precisely, at every point $x\in X$ where the Lelong number $\nu(\varphi,x)=0$ vanishes,
        \item [$(e)$] If $\nu(\varphi,x)=0$ for every $x\in X$ (in particular, if $\varphi$ is locally bounded), then $\lambda_\nu$ converges uniformly to $0$ on every compact subset of $X$. 
    \end{itemize}
\end{theorem}

In \cite[Remark 4.4]{Wat26b}, it was noted that the loss of positivity of \(\varphi_\nu\) along \(E_{+}(\varphi):=\{x\in X\mid\nu(\varphi,x)>0\}\) may persist by a fixed amount; that is, the limit of \(\{\lambda_\nu\}_{\nu\in\bb{N}}\) may remain positive on \(E_{+}(\varphi)\). 
More precisely, the following holds.

\begin{lemma}[{cf. \cite[Lemme 8.6]{Dem82}}]\label{Lemma: Lemme 8.6 in Dem82}
    The limit of the decreasing sequence \(\{\lambda_\nu\}_{\nu\in\bb{N}}\) in Theorem \ref{Theorem: Demailly approximation of smooth functions} is determined by the Lelong number of \(\varphi\) and the Griffiths positivity of the tangent bundle \(T_X\) as follows. 
    For every \(x\in X\), we have
    \begin{align*}
        \lim_{\nu\to+\infty}\lambda_\nu(x)=\tau_{-}(x)\nu(\varphi,x),
    \end{align*}
    where $\tau_{-}(x):=\sup(0,-\tau(x))$, and \(\tau(x)\) denotes the smallest eigenvalue of the curvature of the tangent bundle \(T_X\) in the sense of Griffiths. 
    More precisely, for a smooth Hermitian metric \(h_X\) on \(T_X\), the function $\tau(x)$ is defined by 
    \begin{align*}
        \tau(x)=\frac{1}{2\pi}\inf_{\substack{|\xi|=|\eta|=1\\\xi,\eta\in\bb{C}^n}}c_{jk\ell m}\xi_j\overline{\xi}_k\eta_\ell\overline{\eta}_m=\inf_{\substack{|\xi|=|\eta|=1\\\xi,\eta\in T_{X,x}}}\frac{i}{2\pi}\Theta_{T_X,h_X}(\xi\otimes\eta,\xi\otimes\eta),
    \end{align*}
    where \(c_{jk\ell m}\) denotes the coefficients of the curvature tensor $\iO{T_X,h_X}$ with respect to an orthonormal basis of $T_{X,x}$.
    
    Furthermore, if $T_X$ is Griffiths semi-positive along $E_{+}(\varphi)$, then $\{\lambda_\nu\}_{\nu\in\bb{N}}$ decreases to $0$ on $X$. 
    In particular, if $T_X$ is Griffiths semi-positive on $X$, then for any approximation obtained in Theorem \ref{Theorem: Demailly approximation of smooth functions}, the decreasing sequence $\{\lambda_\nu\}_{\nu\in\bb{N}}$ converges to $0$ on $X$.
\end{lemma}

This estimate in terms of the Griffiths positivity of the tangent bundle $T_X$ was further generalized in \cite{Dem92}, where a Demailly-type approximation is obtained along the sublevel sets of the Lelong number.
In particular, an important application involving Lemma \ref{Lemma: Lemme 8.6 in Dem82} is that the pseudo-effective cone \(H^{1,1}_{\rom{psef}}(X)\) coincides with the nef cone \(H^{1,1}_{\rom{nef}}(X)\) if the tangent bundle \(T_X\) is Griffiths semi-positive on \(X\) (in fact, nefness suffices) (see \cite[Corollary 1.5]{Dem92}).

\section{Proofs of the Main Results}\label{Section: Proofs}

In this section, we prove Theorems \ref{Theorem: Hard Lefschetz if TX>0 on D} and \ref{Theorem: Hard Lefschetz for semi-positivity} under the following setup.

\begin{setup}\label{Setup}
    Let $X$ be a compact \kah manifold of dimension $n$ with a \kah metric $\omega$ and $D=\sum^J_{j=1}D_j$ be a simple normal crossing divisor on $X$. 
    Let $\sigma_j$ be the defining section of $D_j$ and $h_j$ be a smooth Hermitian metric on $\Ox{D_j}$. 
    Let $L\longrightarrow X$ be a pseudo-effective line bundle, that is, $L$ admits a singular Hermitian metric $h$ with a semi-positive curvature current on $X$, i.e., $\iO{L,h}\geq0$ in the sense of currents on $X$. 
    By Demailly's approximation theorem (Theorem \ref{Theorem: Demailly approximation preserves ideal sheaves}), there exists a sequence $\{h_\nu\}_{\nu\in\bb{N}}$ of singular Hermitian metrics on $L$ satisfying the following properties: 
    \begin{itemize}
        \item [($i$)] $h_\nu$ is smooth on $X\setminus Z_\nu$, where $Z_\nu$ is an analytic subset of $X$ with $Z_\nu\subset Z_{\nu+1}$,
        \item [($ii$)] $\{h_\nu\}_{\nu\in\bb{N}}$ is increasing and converges to $h$, i.e., $h_\nu\nearrow h$ on $X$,
        \item [($iii$)] $\scr{I}(h)=\scr{I}(h_\nu)$ on $X$ for all $\nu\in\bb{N}$,
        \item [($iv$)] $\iO{L,h_\nu}\geq-\varepsilon_\nu\omega$ on $X$, where $\{\varepsilon_\nu\}_{\nu\in\bb{N}}$ is a decreasing sequence of nonnegative real numbers converging to $0$. 
    \end{itemize}
    
    We take a smooth \kah metric $\omega_P$ on $X\setminus D$ which is of Poincar\'{e}-type along $D$, satisfying $\omega_P\geq\omega$ on $X\setminus D$. Note that $\omega_P$ is complete on $X\setminus D$ (see \cite{Zuc79}).
    For each analytic subset $Z_\nu$ of $X$, there exists a quasi-plurisubharmonic functions $\psi_\nu$ on $X$ such that $\psi_\nu=-\infty$ on $Z_\nu$, $\psi_\nu$ is smooth on $X\setminus Z_\nu$, and $\omega+\idd\psi_\nu>0$ defines a complete \kah metric on $X\setminus Z_\nu$ (see \cite[Th\'eor\`em 1.5]{Dem82}).
    We can find a family 
    \begin{align*}
        \omega_{P,\nu,\delta}:=\omega_P+\delta(\omega+\idd\psi_\nu), \qquad \delta>0
    \end{align*}
    of complete \kah metrics on $X\setminus(D\cup Z_\nu)$. 

    The compactness of $X$ allows the assumption that $\psi_\nu<0$ on $X$. 
    The strong openness property (see \cite{GZ15}) implies that $\scr{I}(h_\nu)=\bigcup_{\rho>0}\scr{I}(h_\nu e^{-\rho\psi_\nu})$ on $X$. 
    By the compactness of $X$ and the strong Noetherian property of coherent sheaves (see \cite[Chapter II, (3.22)]{Dem-book}), there exists $\rho_{X,\nu}>0$ such that $\bigcup_{\rho>0}\scr{I}(h_\nu e^{-\rho\psi_\nu})=\scr{I}(h_\nu e^{-\rho_{X,\nu}\psi_\nu})$ on $X$. 
    Hence, we obtain $\scr{I}(h_\nu)=\scr{I}(h_\nu e^{-\varrho\psi_\nu})$ on $X$ for any $0<\varrho<\rho_{X,\nu}$. 
    The same argument applied to \(h\) instead of \(h_\nu\) shows that there exists some \(\tl{\rho}_X>0\) such that $\scr{I}(h)=\scr{I}(h e^{-\varrho\psi_\nu})$ on $X$ for any $0<\varrho<\tl{\rho}_X$. 
    Furthermore, it follows from the compactness of $X$ that $c_\nu\idd\psi_\nu\geq-\varepsilon_\nu\omega$ on \(X\) in the sense of currents for a sufficiently small \(c_\nu>0\).
    Hence, for any suﬃciently small number \(0<\varrho<\mu_\nu:=\min\{\rho_{X,\nu},\tl{\rho}_X,c_\nu\}\), we have the following properties:
    \begin{itemize}
        \item [$(v)$] $\scr{I}(he^{-\varrho\psi_\nu})=\scr{I}(h)=\scr{I}(h_\nu)=\scr{I}(h_\nu e^{-\varrho\psi_\nu})$ on $X$, 
        \item [$(vi)$] $\varrho\,\idd\psi_\nu\geq-\varepsilon_\nu\omega$ on $X$ in the sense of currents.
    \end{itemize}
\end{setup}

\begin{proof}[Proof of Theorem \ref{Theorem: Hard Lefschetz if TX>0 on D}]
    Let $\hbar_{D_j}:=1/|\sigma_j|^2=h_j/|\sigma_j|^2_{h_j}$ be the natural singular Hermitian metric on $\Ox{D_j}$, whose curvature current is semi-positive. 
    Let $\varphi_j:=\log|\sigma_j|^2_{h_j}$ be a quasi-plurisubharmonic weight function of $\hbar_{D_j}$. 
    Applying Theorem \ref{Theorem: Demailly approximation of smooth functions} to each $\varphi_j$, there exist a decreasing sequence $\{\varphi_{j,\nu}\}_{\nu\in\bb{N}}$ of smooth functions on $X$ 
    and a decreasing sequence $\{\lambda_{j,\nu}\}_{\nu\in\bb{N}}$ of nonnegative continuous functions on $X$ such that, upon setting $h_{j,\nu}:=h_je^{-\varphi_{j,\nu}}$ be a smooth Hermitian metric on $\Ox{D_j}$ and the following hold:
    \begin{itemize}
        \item [($a$)] $\displaystyle\lim_{\nu\to+\infty}\varphi_{j,\nu}=\varphi_j$ on $X$, i.e., $\{h_{j,\nu}\}_{\nu\in\bb{N}}$ is increasing and converges to $\hbar_{D_j}$, 
        \item [($b$)] $\iO{\Ox{D_j},h_{j,\nu}}\geq-\lambda_{j,\nu}\,\omega$ on $X$, 
        \item [($c$)] $\lambda_{j,\nu}\longrightarrow0$ almost everywhere on $X$, more precisely, $\lambda_{j,\nu}\longrightarrow0$ on $X\setminus D$, 
        \item [($d$)] $\lambda_{j,\nu}$ converges uniformly to $0$ on every compact subset of $X\setminus D$. 
    \end{itemize}
    Furthermore, by Lemma \ref{Lemma: Lemme 8.6 in Dem82} and the assumption that $T_X$ is Griffiths semi-positive on $D$, conditions $(c)$ and $(d)$ can be replaced by the following stronger conditions.
    \begin{itemize}
        \item [($c'$)] $\displaystyle\lim_{\nu\to+\infty}\lambda_{j,\nu}(x)=0$ for every $x\in X$, i.e., $\lambda_{j,\nu}\searrow0$ on $X$ as $\nu\to+\infty$,
        \item [($d'$)] letting $\displaystyle \kappa_{j,\nu}\!:=\!\max_{x\in X}\lambda_{j,\nu}(x)\!\geq\!0$, the decreasing sequence $\{\kappa_{j,\nu}\}_{\nu\in\bb{N}}$ converges to $0$.
    \end{itemize}

    For each $\nu\in\bb{N}$, we define a singular Hermitian metric $\hbar^L_{D,\nu}$ on $L|_{X\setminus D}$ by 
    \begin{align*}
        \hbar^L_{D,\nu}:=h_\nu(2\alpha)^{-2\alpha}\prod_{j=1}^{J}|\sigma_j|^2_{h_{j,\nu}}\Big(\log \big(e^{-2\alpha}|\sigma_j|^2_{h_{j,\nu}}\big)\Big)^{2\alpha},
    \end{align*}
    where $\alpha$ is a sufficiently large positive integer depending only on the compactness of $X$ and not on $h_\nu$ or $h_{j,\nu}$. Thus, each $\hldnu$ can be chosen to satisfy Theorem \ref{Theorem: Logarithmic L2-Dolbeault resolusion}.
    Lemma \ref{Lemma: hldnu increasing to h} follows from Theorem \ref{Theorem: Logarithmic L2-Dolbeault resolusion}, condition $(a)$ for the sequence \(\{\varphi_{j,\nu}\}_{\nu\in\bb{N}}\) and Lemma \ref{Lemma: function of log}. 

    \begin{lemma}\label{Lemma: hldnu increasing to h}
        A sequence $\{|\sigma_j|^2_{h_{j,\nu}}\}_{\nu\in\bb{N}}$ of smooth functions on $X$ and the sequence $\{\hldnu\}_{\nu\in\bb{N}}$ of singular Hermitian metrics on $L|_{X\setminus D}$ satisfy the following properties.
        \begin{itemize}
            \item $0\leq|\sigma_j|^2_{h_{j,\nu}}\leq1$ on $X$ for every $\nu\in\bb{N}$,
            \item 
            $|\sigma_j|^2_{h_{j,\nu}}\nearrow1$ on $X$ as $\nu\to+\infty$, 
            \item $\{\hldnu\}_{\nu\in\bb{N}}$ is increasing and converges to $h$, i.e., $\hldnu\nearrow h$ on $X$, 
            \item each $\nu$ and any $p\geq0$, the complex $\Bigl(\scr{L}^{p,\ast}_{L,\hldnu,\omega_P},\dbar\Bigr)$ is logarithmic $L^2$-Dolbeault fine resolusion of $\Omega^p_X(\log D)\otimes\Ox{L}\otimes\scr{I}(h_\nu)$. 
            In particular, we have 
            \begin{align*}
                \Omega^p_X(\log D)\otimes\Ox{L}\otimes\scr{I}(h_\nu)=\rom{Ker}\Bigl(\,\dbar:\scr{L}^{p,0}_{L,\hldnu,\omega_P}\longrightarrow\scr{L}^{p,1}_{L,\hldnu,\omega_P}\Bigr).
            \end{align*}
        \end{itemize}
    \end{lemma}


    \begin{lemma}\label{Lemma: function of log}
        For any positive integer $\alpha\in\bb{N}$, the function 
        \begin{align*}
            x\Big(\log \big(e^{-2\alpha}x\big)\Big)^{2\alpha}
        \end{align*}
        is strictly increasing on $(0,1)$. Furthermore, we obtain 
        \begin{align*}
        1+\frac{2\alpha}{\log \big(e^{-2\alpha}x\big)}=\frac{\log x}{\log x-2\alpha}>0    
        \end{align*}
        for $0<x<1$, and 
        \begin{align*}
        \sup_{0<x<1}\biggl(1+\frac{2\alpha}{\log \big(e^{-2\alpha}x\big)}\biggr)=1.
        \end{align*}
    \end{lemma}

    
    Let $\{\beta\}\in H^q(X,\Omega^n_X(\log D)\otimes L\otimes\scr{I}(h))=H^q(X,K_X\otimes\Ox{D}\otimes L\otimes\scr{I}(h))$ be a cohomology class represented by a smooth form with values in $\Omega^n_X(\log D)\otimes L\otimes\scr{I}(h)=K_X\otimes\Ox{D}\otimes L\otimes\scr{I}(h)$. 
    Let $h^D_{C^\infty}$ be a smooth Hermitian metric on $\Ox{D}$.
    By Theorem \ref{Theorem: L2-Dolbeault resolusion} and condition $(v)$, the integrability follows for every $0\leq\varrho<\mu_\nu$. 
    \begin{align*}
        \int_X|\beta|^2_{\wedge^{n,q}\omega\otimes h\otimes h^D_{C^\infty}}e^{-\varrho\psi_\nu}dV_\omega<+\infty.
    \end{align*}
    Using the smoothness of $h^D_{C^\infty}$ on the compact space $X$ and applying Lemma \ref{Lemma: inequality of (n,0)-forms} to $\omega_P\geq\omega$, there exists $C>0$ such that
    \begin{align*}
        ||\beta||^2_{\varrho}:=\int_{X\setminus D}|\beta|^2_{\wedge^{n,q}\omega_P\otimes h}e^{-\varrho\psi_\nu}dV_{\omega_P}&\leq\int_{X\setminus D}|\beta|^2_{\wedge^{n,q}\omega\otimes h}e^{-\varrho\psi_\nu}dV_\omega\\
        &\leq C\int_X|\beta|^2_{\wedge^{n,q}\omega\otimes h\otimes h^D_{C^\infty}}e^{-\varrho\psi_\nu}dV_\omega<+\infty.
    \end{align*}
    Note that \(\varrho\) is allowed to be \(0\), and set $||\beta||^2:=||\beta||^2_0$.
    Since $e^{-\varrho\psi_\nu}\searrow1$ pointwise almost everywhere on $X\setminus D$ as \(\varrho\to0\), the Lebesgue monotone convergence theorem yields $||\beta||^2_\varrho\searrow||\beta||^2$. 
    Therefore, we take a sufficiently small $\varrho_\nu\in(0,\mu_\nu)$ satisfying $||\beta||^2_{\varrho_\nu}\leq 2||\beta||^2$, and define the singular Hermitian metric $\hldpn$ on $L|_{X\setminus D}$ by adding $\varrho_\nu\psi_\nu$ to $\hldnu$ as a weight function (as required for Lemma \ref{Lemma: log resolusion for (n,q)-forms} below): 
    \begin{align*}
        \hldpn:=\hldnu e^{-\varrho_\nu\psi_\nu}=h_\nu e^{-\varrho_\nu\psi_\nu}(2\alpha)^{-2\alpha}\prod_{j=1}^{J}|\sigma_j|^2_{h_{j,\nu}}\Big(\log \big(e^{-2\alpha}|\sigma_j|^2_{h_{j,\nu}}\big)\Big)^{2\alpha}.
    \end{align*} 
    By Lemmas \ref{Lemma: inequality of (n,0)-forms} and \ref{Lemma: hldnu increasing to h} and the monotonicity of $\opnd$, we obtain 
    \begin{align*}
        ||\beta||^2_{\nu,\psi,\delta}:=&\int_{X\setminus(D\cup Z_\nu)}\hspace{-4mm}|\beta|^2_{\wedge^{n,q}\opnd\otimes\hldpn}dV_{\opnd}=\int_{X\setminus(D\cup Z_\nu)}\hspace{-4mm}|\beta|^2_{\wedge^{n,q}\opnd\otimes\hldnu}e^{-\varrho_\nu\psi_\nu}dV_{\opnd}\\
        \leq&\,||\beta||^2_{\varrho_\nu}\!=\!\int_{X\setminus D}|\beta|^2_{\wedge^{n,q}\omega_P\otimes h}e^{-\varrho_\nu\psi_\nu}dV_{\omega_P}
        \leq2||\beta||^2\!=\!2\!\int_{X\setminus D}|\beta|^2_{\wedge^{n,q}\omega\otimes h}dV_\omega<+\infty.
    \end{align*} 
    Here, it should be noted that $\opnd$ is of \Pt along $D$ but is not necessarily smooth on $X\setminus D$. 
    Define the sheaf $\scr{L}^{p,q}_{L,\hldpn,\opnd}$ in the same way as $\scr{L}^{p,q}_{L,\hldnu,\omega_P}$. 
    In particular, we obtain $\Gamma\Big(X\setminus (D\cup Z_\nu),\scr{L}^{n,q}_{L,\hldpn,\omega_{P,\nu,\delta}}\Big)$
    \begin{align*}
        =\!\biggl\{u\!:\!X\!\setminus\! D\longrightarrow\!\bigwedge^{n,q}\!T^*_X\!\otimes\! L\,\bigg|\int_{X\setminus (D\cup Z_\nu)}\hspace{-5mm}\big(|u|^2_{\wedge^{n,q}\opnd\otimes\hldpn}\!+\!\big|\dbar u\big|^2_{\wedge^{n,q+1}\opnd\otimes\hldpn}\big)dV_{\opnd}\!\!<\!+\infty\!\biggr\}.
    \end{align*}

    \begin{lemma}\label{Lemma: log resolusion for (n,q)-forms}
        For all $\nu\in\bb{N}$ and $\delta\geq0$, the complex $\Big(\scr{L}^{n,\ast}_{L,\hldpn,\omega_{P,\nu,\delta}},\dbar\Big)$ is logarithmic $L^2$-Dolbeault resolusion of $\Omega^n_X(\log D)\otimes L\otimes\scr{I}(h_\nu e^{-\varrho_\nu\psi_\nu})=K_X\otimes\Ox{D}\otimes L\otimes\scr{I}(h)$.
        Thus, for any $q\geq0$, we have the logarithmic $L^2$-Dolbeault isomorphism 
        \begin{align*}
            H^q(X,\Omega^n_X(\log D)\otimes L\otimes\scr{I}(h))\cong H^q\Bigl(\Gamma\Big(X\setminus(D\cup Z_\nu),\scr{L}^{n,\ast}_{L,\hldpn,\omega_{P,\nu,\delta}}\Big)\Bigr).
        \end{align*}
    \end{lemma}

    The proof of Lemma \ref{Lemma: log resolusion for (n,q)-forms} is given after this proof.
    Since $\Gamma\Big(X\setminus(D\cup Z_\nu),\scr{L}^{n,q}_{L,\hldpn,\omega_{P,\nu,\delta}}\Big)$ is a Hilbert space, its closed subspace
    \begin{align*}
        &Z^q\Big(\scr{L}^{n,\ast}_{L,\hldpn,\omega_{P,\nu,\delta}}\Big)=\rom{Ker}\Big(\,\dbar:\scr{L}^{n,q}_{L,\hldpn,\omega_{P,\nu,\delta}}\longrightarrow\scr{L}^{n,q+1}_{L,\hldpn,\omega_{P,\nu,\delta}}\Big)\quad \text{on}\quad X\setminus(D\cup Z_\nu)\\
        &\qquad=\!\biggl\{u\!:\!X\!\setminus\!(D\cup Z_\nu)\!\longrightarrow\!\bigwedge^{n,q}\!T^*_X\!\otimes\! L\,\bigg|\,\,\!\dbar u=0,\!\int_{X\setminus(D\cup Z_\nu)}\hspace{-4mm}|u|^2_{\wedge^{n,q}\opnd\otimes\hldpn}dV_{\opnd}\!<\!+\infty\!\biggr\}
    \end{align*}
    is clearly a Hilbert space. 
    By Lemma \ref{Lemma: log resolusion for (n,q)-forms} and the proof of the de Rham-Weil isomorphism, the map $\alpha\mapsto\{\alpha\}$ from the cocycle space $Z^q\Big(\scr{L}^{n,\ast}_{L,\hldpn,\omega_{P,\nu,\delta}}\Big)$ equipped with its $L^2$-topology, into $H^q(X,\Omega^n_X(\log D)\otimes L\otimes\scr{I}(h))$ equipped with its finite vector space topology, is continuous. 
    Banach's open mapping theorem implies that the coboundary space $B^q\Big(\scr{L}^{n,\ast}_{L,\hldpn,\omega_{P,\nu,\delta}}\Big)$ is closed in $Z^q\Big(\scr{L}^{n,\ast}_{L,\hldpn,\omega_{P,\nu,\delta}}\Big)$. This is true for all $\delta\geq0$ and the limit case $\delta = 0$ yields the strongest $L^2$-topology in bidegree $(n,q)$. 
    Thus, the Hilbert space $Z^q\Big(\scr{L}^{n,\ast}_{L,\hldpn,\omega_{P,\nu,\delta}}\Big)$ admits an orthogonal decomposition, and \(\beta\) is a \(\dbar\)-closed form in the Hilbert space induced by the complete \kah metric \(\opnd\) on \(X\setminus (D\cup Z_\nu)\). Therefore, there exists an \(\opnd\)-harmonic form \(u_{\nu,\delta}\) representing the same cohomology class as \(\beta\), satisfying 
    \begin{align*}
        ||u_{\nu,\delta}||^2_{\nu,\psi,\delta}\leq||\beta||^2_{\nu,\psi,\delta} \quad\text{and}\quad \overline{\Box}_{\nu,\delta}u_{\nu,\delta}=\big(\dbar^*_{\nu,\delta}\dbar+\dbar\dbar^*_{\nu,\delta}\big) u_{\nu,\delta}=0,
    \end{align*}
    where $\dbar^*_{\nu,\delta}=\dbar^*_{\hldpn,\opnd}$. The existence of a harmonic representative holds true only for $\delta > 0$, because we need to have a complete \kah metric on $X\setminus (D\cup Z_\nu)$.

    Let $v_{\nu,\delta}$ be the unique $(n-q,0)$-form such that $u_{\nu,\delta}=v_{\nu,\delta}\wedge\opnd^q$, whose existence follows from the pointwise Lefschetz isomorphism. 
    Then, we obtain
    \begin{align*}
        (q!)^2||v_{\nu,\delta}||^2_{\nu,\psi,\delta}=||u_{\nu,\delta}||^2_{\nu,\psi,\delta}\leq||\beta||^2_{\nu,\psi,\delta}\leq2||\beta||^2,
    \end{align*}
    where \(u_{\nu,\delta}\) is primitive as a \((0,n-q)\)-form. 
    A straightforward computation shows that
    \begin{align*}
        \iO{L,\hldpn}&=\iO{L,h_\nu}+\varrho_\nu\idd\psi_\nu+i\sum^J_{j=1}\Theta_{\Ox{D_j},h_{j,\nu}}\\
        &\qquad+i\sum^J_{j=1}\frac{2\alpha\Theta_{\Ox{D_j},h_{j,\nu}}}{\log\big(e^{-2\alpha}|\sigma_j|^2_{h_{j,\nu}}\big)}+i\sum^J_{j=1}\frac{2\alpha\partial\log|\sigma_j|^2_{h_{j,\nu}}\wedge\dbar\log|\sigma_j|^2_{h_{j,\nu}}}{\Big(\log\big(e^{-2\alpha}|\sigma_j|^2_{h_{j,\nu}}\big)\Big)^2}\\
        &\geq\iO{L,h_\nu}+\varrho_\nu\idd\psi_\nu+i\sum^J_{j=1}\biggl(1+\frac{2\alpha}{\log\big(e^{-2\alpha}|\sigma_j|^2_{h_{j,\nu}}\big)}\biggr)\Theta_{\Ox{D_j},h_{j,\nu}}\\
        &\geq-\Biggl(2\varepsilon_\nu+\sum^J_{j=1}\biggl(1+\frac{2\alpha}{\log\big(e^{-2\alpha}|\sigma_j|^2_{h_{j,\nu}}\big)}\biggr)\lambda_{j,\nu}\Biggr)\omega\\
        &\geq-\Bigl(2\varepsilon_\nu+\sum^J_{j=1}\lambda_{j,\nu}\Bigr)\omega\\
        &\geq-\Bigl(2\varepsilon_\nu+\sum^J_{j=1}\lambda_{j,\nu}\Bigr)\opnd,
    \end{align*}
    using Lemma \ref{Lemma: function of log} and conditions $(vi)$ and $(b)$.
    Since $\hldpn$ is smooth and $\opnd$ is complete on $X\setminus(D\cup Z_\nu)$, and the curvature $\iO{L,\hldpn}$ is bounded from below, we can therefore apply Proposition \ref{Proposition: Prop in DPS01} to $\big(X\setminus(D\cup Z_\nu),\opnd,L,\hldpn\big)$, and the Bochner formula yields 
    \begin{align*}
        (q!)^2\big|\big|\dbar v_{\nu,\delta}\big|\big|^2_{\nu,\psi,\delta}&\leq \frac{q}{q+1}\int_{X\setminus(D\cup Z_\nu)}\Bigl(2\varepsilon_\nu+\sum^J_{j=1}\lambda_{j,\nu}\Bigr)|u_{\nu,\delta}|^2_{\wedge^{n,q}\opnd\otimes\hldpn}dV_{\opnd}\\
        &\leq \frac{q}{q+1}\Bigl(2\varepsilon_\nu+\sum^J_{j=1}\kappa_{j,\nu}\!\Bigr)||u_{\nu,\delta}||^2_{\nu,\psi,\delta}\leq \frac{2q}{q+1}\Bigl(2\varepsilon_\nu+\sum^J_{j=1}\kappa_{j,\nu}\Bigr)||\beta||^2.
    \end{align*}
    These uniform bounds imply that there are subsequences $u_{\nu,\delta_k}$ and $v_{\nu,\delta_k}$ with $\delta_k\to 0$, possessing weak-$L^2$ limits $u_\nu = \lim_{k\to+\infty}u_{\nu,\delta_k}$ and $v_\nu = \lim_{k\to+\infty}v_{\nu,\delta_k}$. 
    For simplicity, we write $L^2_{p,q}(\opnd):=L^2_{p,q}\bigl(X\setminus(D\cup Z_\nu),L;\opnd,\hldpn\bigr)$.
    Here $v_\nu = \lim_{k\to+\infty}v_{\nu,\delta_k}$ is the limit with respect to $L^2_{0,n-q}(\omega_P):=L^2_{0,n-q}\bigl(X\setminus(D\cup Z_\nu),L;\omega_P,\hldpn\bigr)=L^2_{0,n-q}(\omega_{P,\nu,0})$. 
    This follows from Lemma \ref{Lemma: inequality of (n,0)-forms}, since $L^2_{0,n-q}(\omega_P)$ has the weakest topology among all $L^2_{0,n-q}(\opnd)$ spaces.
    On the other hand \(u_\nu=\lim_{k\to+\infty}u_{\nu,\delta_k}\) exists in every space \(L^2_{n,q}(\opnd)\) for \(\delta>0\). This is because the topology becomes stronger as \(\delta\searrow0\). 
    (However this convergence does not necessarily hold in \(L^2_{n,q}(\omega_P)\). Indeed in bidegree \((n,q)\) the topology of \(L^2_{n,q}(\omega_P)\) can be strictly stronger than that of all spaces \(L^2_{n,q}(\opnd)\).)

    The above estimates yield 
    \begin{align*}
        (q!)^2||v_\nu||^2_\nu
        &=||u_\nu||^2_\nu:=\int_{X\setminus D}|u_\nu|^2_{\wedge^{n,q}\omega_P\otimes\hldnu}dV_{\omega_P}\\
        &\leq||u_\nu||^2_{\nu,\psi}:=\int_{X\setminus D}|u_\nu|^2_{\wedge^{n,q}\omega_P\otimes\hldnu}e^{-\varrho_\nu\psi_\nu}dV_{\omega_P}\leq 2||\beta||^2,\\
        \big|\big|\dbar v_\nu\big|\big|^2_\nu&\leq\frac{2}{(q+1)!(q-1)!}\Bigl(2\varepsilon_\nu+\sum^J_{j=1}\kappa_{j,\nu}\Bigr)||\beta||^2,\\
        u_\nu&=\omega^q_P\wedge v_\nu\equiv\beta \qquad \text{in}\quad H^q(X,\Omega^n_X(\log D)\otimes L\otimes\scr{I}(h_\nu e^{-\varrho_\nu\psi_\nu})).
    \end{align*}
    The last equality can be checked via the de Rham-Weil isomorphism, by using condition $(v)$ and the fact that the map $Z^q\Big(\scr{L}^{n,\ast}_{L,\hldpn,\omega_{P,\nu,\delta}}\Big)\!\longrightarrow\! H^q(X,\Omega^n_X(\log D)\otimes L\otimes\scr{I}(h)):\alpha\mapsto\{\alpha\}$ is continuous. 
    Again, using the monotonicity of $\{\hldnu\}_{\nu\in\bb{N}}$ from Lemma \ref{Lemma: hldnu increasing to h} and considering a fixed given Hilbert space $L^2(\hbar^L_{D,\nu_0})$, we find $L^2$ convergent subsequences $u_\nu\to u$, $v_\nu\to v$ as $\nu\to+\infty$. 
    It follows from conditions $(iii)$ and $(d')$ above that 
    \begin{align*}
        \big|\big|\dbar v_\nu\big|\big|^2_\nu\leq \frac{2}{(q+1)!(q-1)!}\Bigl(2\varepsilon_\nu+\sum^J_{j=1}\kappa_{j,\nu}\Bigr)||\beta||^2\longrightarrow 0
    \end{align*} 
    as $\nu\to+\infty$, and together with the $L^2_{loc}$-convergence $v_\nu\to v$, we obtain $\dbar v=0$ in the sense of distributions.
    Indeed, for any test form $\phi$, 
    \begin{align*}
        \Big|\big\langle\dbar v_\nu,\phi\big\rangle_{\nu_0}\Big|\leq\big|\big|\dbar v_\nu\big|\big|^2_{\nu_0}||\phi||^2_{\nu_0}\leq\big|\big|\dbar v_\nu\big|\big|^2_\nu||\phi||^2_{\nu_0}\longrightarrow0,
    \end{align*}
    and hence 
    \begin{align*}
        \big\langle\dbar v,\phi\big\rangle_{\nu_0}=\big\langle v,\dbar^*_{\nu_0}\psi\big\rangle_{\nu_0}=\lim_{\nu\to+\infty}\big\langle v_\nu,\dbar^*_{\nu_0}\psi\big\rangle_{\nu_0}=\lim_{\nu\to+\infty}\big\langle\dbar v_\nu,\psi\big\rangle_{\nu_0}=0.
    \end{align*}
    Furthermore, the above inequalities yield 
    \begin{align*}
        ||v||^2&=\int_{X\setminus D}|v|^2_{\wedge^{0,n-q}\omega_P\otimes h}dV_{\omega_P}\leq2||\beta||^2<+\infty,\\
        u&=\omega^q_P\wedge v\equiv\beta\qquad\text{in}\quad H^q(X,\Omega^n_X(\log D)\otimes L\otimes\scr{I}(h)). 
    \end{align*}
    The last equality can also be verified using the de Rham-Weil isomorphism, where the equisingularity property $\scr{I}(h)=\scr{I}(h_\nu e^{-\varrho_\nu\psi_\nu})$, i.e., condition $(v)$, plays a crucial role. 
    In particular, for any integer $\nu\geq\nu_0$, we obtain 
    \begin{align*}
        ||v||^2_\nu=\int_{X\setminus D}|v|^2_{\wedge^{0,n-q}\omega_P\otimes\hldnu}dV_{\omega_P}\leq2||\beta||^2<+\infty.
    \end{align*}
    Finally, the condition on the section \(v\), together with Theorem \ref{Theorem: Logarithmic L2-Dolbeault resolusion} or Lemma \ref{Lemma: hldnu increasing to h}, yields 
    \begin{align*}
        v\in \Gamma\Big(X\setminus D,\scr{L}^{n-q,0}_{L,\hbar^L_{D,\nu_1}\!,\,\omega_P}\Big)\cap\rom{Ker}\,\dbar=H^0(X,\Omega_X^{n-q}(\log D)\otimes L\otimes\scr{I}(h)),
    \end{align*}
    where $\nu_1$ is an integer with $\nu_1\geq\nu_0$.
    This completes the proof.
\end{proof}

\begin{proof}[Proof of Lemma \ref{Lemma: log resolusion for (n,q)-forms}]
    It is enough to prove that the logarithmic $L^2$-Dolbeault complex
    \begin{align*}
        0\longrightarrow \Omega_X^n(\log D)\otimes\cal{O}_X(L)\otimes\scr{I}(h_\nu e^{-\varrho_\nu\psi_\nu})\hookrightarrow\scr{L}^{n,0}_{L,\hldpn,\omega_{P,\nu,\delta}}\overset{\dbar}{\longrightarrow}\scr{L}^{n,1}_{L,\hldpn,\omega_{P,\nu,\delta}}\overset{\dbar}{\longrightarrow}\cdots
    \end{align*}
    is exact. Theorem \ref{Theorem: Logarithmic L2-Dolbeault resolusion} shows that the complex is exact when \(\delta=0\). 
    The issue is that, when \(\delta>0\), additional completeness is imposed on \(\opnd\) along \(Z_\nu\). 
    First, as stated in Lemma \ref{Lemma: inequality of (n,0)-forms}, the product of the norm of \((n,0)\)-forms and the volume form is independent of the choice of metric. 
    Hence, the above complex is exact for \(q=0\), that is, 
    \begin{align*}
        \Omega^n_X(\log D)\otimes\Ox{L}\otimes\scr{I}(h_\nu e^{-\varrho_\nu\psi_\nu})=\rom{Ker}\Big(\,\dbar:\scr{L}^{n,0}_{L,\hldpn,\omega_{P,\nu,\delta}}\longrightarrow\scr{L}^{n,1}_{L,\hldpn,\omega_{P,\nu,\delta}}\Big).
    \end{align*}

    We prove the exactness of the above complex for \(q\geq 1\). As is well known, the exactness on \(X\setminus (D\cup Z_\nu)\) follows from the \(L^2\)-estimates for $\big(\hldpn,\opnd\big)$, obtained by adding a smooth strictly plurisubharmonic function as a weight. 
    This is because the completeness of $\opnd$ causes no difficulty on \(X\setminus (D\cup Z_\nu)\). The main issue is therefore the exactness at a point \(x_0\in D\cup Z_\nu\). To this end, we take a sufficiently small Stein neighborhood \(U\) of \(x_0\) and prove that the \(L^2\)-estimates hold on \(U\) with respect to $\big(\hldpn,\opnd\big)$. 
    Here, the fact that \(\psi_\nu\) has been added as a weight plays a crucial role. 

    Let $(z_1,\ldots,z_n)$ be local coordinates on $U$ chosen to satisfy $\idd|z|^2\geq\omega$, and set $U^*:=U\setminus(D\cup Z_\nu)$ which is also Stein. 
    For simplicity, assume that $0<\delta<\varrho_\nu<1$. Since the argument is local, following \cite{HLWY23,Wat26a}, we may assume that
    \begin{align*}
        c\,\idd|z|^2+i\sum^J_{j=1}\frac{2\alpha\partial\log|\sigma_j|^2_{h_{j,\nu}}\wedge\dbar\log|\sigma_j|^2_{h_{j,\nu}}}{\big(\log\big(e^{-2\alpha}|\sigma_j|^2_{h_{j,\nu}}\big)\big)^2}\geq\omega_P
    \end{align*}
    on $U\setminus D$ for some $c>0$. Let $\tau_\nu:=\varepsilon_\nu+\sum^J_{j=1}\kappa_{j,\nu}\!>\!0$, and write $\hbar^{\psi,\tau}_{D,\nu}:=\hldpn\,e^{-(1+c+\tau_\nu)|z|^2}$. 
    Then $\hbar^{\psi,\tau}_{D,\nu}\sim\hldpn$ on $U^*$, and we obtain 
    \begin{align*}
        \iO{L,\hbar^{\psi,\tau}_{D,\nu}}&=\iO{L,h_\nu}+\varrho_\nu\idd\psi_\nu+(1+c+\tau_\nu)\idd|z|^2+i\sum_{j\in J}\Theta_{\Ox{D_j},h_{j,\nu}}\\
        &\qquad+i\sum_{j\in J}\frac{2\alpha\Theta_{\Ox{D_j},h_{j,\nu}}}{\log\big(e^{-2\alpha}|\sigma_j|^2_{h_{j,\nu}}\big)}+i\sum_{j\in J}\frac{2\alpha\partial\log|\sigma_j|^2_{h_{j,\nu}}\wedge\dbar\log|\sigma_j|^2_{h_{j,\nu}}}{\big(\log\big(e^{-2\alpha}|\sigma_j|^2_{h_{j,\nu}}\big)\big)^2}\\
        &\geq-\Big(\varepsilon_\nu+\sum_{j\in J}\lambda_{j,\nu}\Big)\omega+(1+\tau_\nu)\omega+\delta\idd\psi_\nu+\omega_P\\
        &\geq\omega_P+\omega+\delta\idd\psi_\nu\geq\opnd
    \end{align*}
    on $U^*$. Hence, $\big[\iO{L,\hbar^{\psi,\tau}_{D,\nu}},\Lambda_{\opnd}\big]\geq q\cdot\rom{Id}_L$ holds on $\bigwedge^{n,q}T^*_X\otimes L|_{U^*}$, and thus it is uniformly positive. 
    By $L^2$-estimates (see \cite[\!VIII, \!Theorem 5.2]{Dem-book}), for any $\dbar$-closed section 
    \begin{align*}
        f\in \Gamma\Big(U^*,\scr{L}^{n,q}_{L,\hldpn,\opnd}\Big)\cap\rom{Ker}\,\dbar=L^2_{n,q}\big(U^*,L;\opnd,\hldpn\big)=L^2_{n,q}\big(U^*,L;\opnd,\hdt\big),
    \end{align*}
    there exists $u\in L^2_{n,q-1}\big(U^*,L;\opnd,\hdt\big)$ such that $\dbar u=f$ on $U^*$ and 
    \begin{align*}
        \int_{U^*}|u|^2_{\hdt,\opnd}dV_{\opnd}&\leq\int_{U^*}\big\langle\big[\iO{L,\hbar^{\psi,\tau}_{D,\nu}},\Lambda_{\opnd}\big]^{-1}f,f\big\rangle_{\hdt,\opnd}dV_{\opnd} \tag*{$(\ast)$}\\
        &\leq\frac{1}{q}\int_{U^*}|f|^2_{\hdt,\opnd}dV_{\opnd}<+\infty.
    \end{align*}
    In particular, since $u\in\Gamma\Big(U^*,\scr{L}^{n,q-1}_{L,\hldpn,\opnd}\Big)$, the exactness of the above complex follows.

    Replacing $c>0$ if necessary, we still have $\hdt\sim\hldpn$ on \(U^*\), and obtain $L^2(U^*,\opnd)$ $=L^2(U^*,C\opnd)$ for any \(C>0\). 
    In general, even without the above assumptions on $c$ and $\delta$, there exists a constant $\wp>0$ such that $\iO{L,\hdt}\geq\wp\,\opnd$ on $U^*$, and the same argument applies.
\end{proof}

As a remark, adding $\psi_\nu$ to $\hldnu$ as a weight is necessary for the exactness of the complex in Lemma \ref{Lemma: log resolusion for (n,q)-forms}. 
Indeed, even without \(\psi_\nu\), the curvature of $\hldnu e^{-(1+c+\tau_\nu)|z|^2}$ can be made positive on \(U^*\) locally by the same argument as above. 
However, it cannot be bounded below by \(\wp\,\opnd\) for some \(\wp>0\). Equivalently, its bracket with \(\Lambda_{\opnd}\) cannot be uniformly positive, but degenerates to zero as one approaches \(Z_\nu\).
This is because, due to the completeness of \(\opnd\) along \(Z_\nu\), the eigenvalues of \(\iO{L,\hldnu}\) with respect to \(\opnd\) degenerate to zero as one approaches \(Z_\nu\).
It is therefore not clear whether the boundedness in \((\ast)\) follows solely from the integrability of being a local section of \(\scr{L}^{n,q}_{L,\hldnu,\opnd}\), and hence it is also unclear whether the \(L^2\)-estimate can be applied with respect to \(\opnd\).
In this respect, there appears to be a gap in the proof of the exactness of the complex \(\big(\cal{K}^q_{\varepsilon,\delta},\dbar\big)\) in \cite[Theorem 2.1.1]{DPS01}. In this paper, we overcome this difficulty by using the strong openness property \cite{GZ15}, as in the proof of Theorem \ref{Theorem: Hard Lefschetz if TX>0 on D}.

\vspace{3mm}

We consider removing the assumption on $D$ in Theorem \ref{Theorem: Hard Lefschetz if TX>0 on D} by using the following map induced by multiplication with $\sigma_D$:
\begin{align*}
    \times\sigma_D: H^q(X,K_X\otimes L\otimes\scr{I}(h))\longrightarrow H^q(X,K_X\otimes\Ox{D}\otimes L\otimes\scr{I}(h)).
\end{align*}
In this proof of Theorem \ref{Theorem: Hard Lefschetz if TX>0 on D}, even without assuming the positivity of the tangent bundle $T_X$ on \(D\), the function $2\varepsilon_\nu+\sum_J\lambda_{j,\nu}$ decreases to \(0\) on \(X\setminus D\) as \(\nu\to+\infty\), and its limit on \(D\) is also finite. 
It is therefore natural to ask whether the surjectivity of \(\Phi^q_{\omega_P,h}\) still holds without this assumption. However, this seems unlikely, since the singularities of the natural singular Hermitian metric on \(\Ox{D}\) are not adequately captured. 
Nevertheless, if \(\beta\) and each \(u_{\nu,\delta}\) can be written as $u_{\nu,\delta}=\sigma_D\tl{u}_{\nu,\delta}$ where $\sigma_D=\bigotimes_{j\in J}\sigma_j$, then we obtain 
\begin{align*}
    \hspace{6mm}&\hspace{-6mm}\int_{X\setminus(D\cup Z_\nu)}\!\!\Bigl(2\varepsilon_\nu\!+\!\sum_{j\in J}\lambda_{j,\nu}\!\Bigr)|u_{\nu,\delta}|^2_{\wedge^{n,q}\opnd\otimes\hldpn}dV_{\opnd}\\
    &\leq\sup_{x\in X\setminus D}\biggl(\!\!\Bigl(2\varepsilon_\nu\!+\!\sum_{j\in J}\lambda_{j,\nu}(x)\!\Bigr)|\sigma_D(x)|^2_{h^D_{C^\infty}}\!\!\biggr)\!\int_{X\setminus(D\cup Z_\nu)}|\tl{u}_{\nu,\delta}|^2_{\wedge^{n,q}\opnd\otimes h^{D*}_{C^\infty}\otimes\hldpn}dV_{\opnd}\\
    &\leq \,C^{-1}\!\!\sup_{x\in X\setminus D}\biggl(\!\!\Bigl(2\varepsilon_\nu\!+\!\sum_{j\in J}\lambda_{j,\nu}(x)\!\Bigr)|\sigma_D(x)|^2_{h^D_{C^\infty}}\!\!\biggr)\!\int_{X\setminus(D\cup Z_\nu)}|\tl{u}_{\nu,\delta}|^2_{\wedge^{n,q}\opnd\otimes h}e^{-\varrho_\nu\psi_\nu}dV_{\opnd}.
\end{align*}
Let $\tl{\tau}_{\sigma,\nu}:=\sup_{x\in X\setminus D}\Big(\!\Big(2\varepsilon_\nu+\sum_{j\in J}\lambda_{j,\nu}(x)\Big)|\sigma_D(x)|^2_{h^D_{C^\infty}}\Big)$ for simplicity. 
Furthermore, if $\tl{u}_{\nu,\delta}$ is an $\opnd$-harmonic representative of $\{\beta/\sigma_D\}$, then the following follows.
\begin{align*}
    \tl{\tau}_{\sigma,\nu}\int_{X\setminus(D\cup Z_\nu)}|\tl{u}_{\nu,\delta}|^2_{\wedge^{n,q}\opnd\otimes h}e^{-\varrho_\nu\psi_\nu}dV_{\opnd}
    \leq2\tl{\tau}_{\sigma,\nu}\bigg|\bigg|\frac{\beta}{\sigma_D}\bigg|\bigg|^2.
\end{align*}
Since $\tl{\tau}_{\sigma,\nu}$ decreases to \(0\) as \(\nu\to+\infty\), we can obtain \(\bar\partial v=0\) in the same way as in the proof.
Based on the above observations, we propose the following conjecture.

\begin{conjecture}\label{Conjecture: Conjecture without assumption of D}
    Let $X,D,\omega_P,L,h$ be as in the setting of Theorem \ref{Theorem: Hard Lefschetz if TX>0 on D}. Then, does the wedge multiplication operator $\omega_P^q\wedge\bullet$ induce a surjective morphism 
    \begin{align*}
        \Phi^q_{\omega_P,h}\!:\!H^0(X,\Omega_X^{n-q}(\log D)\otimes L\otimes\scr{I}(h))\longrightarrow H^q(X,\Omega_X^n(\log D)\otimes L\otimes \scr{I}(h))\cap\rom{Im}(\times\sigma_D)
    \end{align*}
    for every nonnegative integer $q$?
\end{conjecture}

The following follow from Lemma \ref{Lemma: Lemme 8.6 in Dem82} or condition $(d')$ in the proof of Theorem \ref{Theorem: Hard Lefschetz if TX>0 on D}. 

\begin{remark}\label{Remark: TX Grif>0 and nefness}
    The assumption that $T_X$ is Griffiths semi-positive on $D$ is stronger than the assumption that $D$ is nef.
\end{remark}

\begin{corollary}\label{Corollary: Hard Lefschetz for h_min}
    Let $X,D,\omega_P,L,h$ be as in the setting of Theorem \ref{Theorem: Hard Lefschetz if TX>0 on D}.
    If each pseudo-effective line bundle $\Ox{D_j}$ admits a singular Hermitian metric $h^{D_j}_{min}$ with minimal singularities such that $E_{+}(h^{D_j}_{min})=\emptyset$, then the wedge multiplication operator $\omega_P^q\wedge\bullet$ induces a surjective morphism 
    \begin{align*}
        \Phi^q_{\omega_P,h}:H^0(X,\Omega_X^{n-q}(\log D)\otimes L\otimes\scr{I}(h))\longrightarrow H^q(X,\Omega_X^n(\log D)\otimes L\otimes \scr{I}(h))
    \end{align*}
    for every nonnegative integer $q$.
\end{corollary}

By \cite[Proposition 1.7]{DPS01}, if $\Ox{D}$ is nef and big, then the minimal singular Hermitian metric \(h_{\min}\) satisfies \(E_{+}(h_{\min})=\emptyset\). 
However, it should be noted that in this case the positivity is sufficiently strong that the cohomology group in the target of \(\Phi^q_{\omega_P,h}\) vanishes by Nadel vanishing theorem, 
and hence this corollary becomes trivial.

\begin{proof}[Proof of Theorem \ref{Theorem: Hard Lefschetz for semi-positivity}]
    By the assumption, we may assume that the curvature $\iO{\Ox{D_j},h_j}$ is semi-positive on $X$.
    We define a singular Hermitian metric $\hbar^L_D$ on $L|_{X\setminus D}$ by 
    \begin{align*}
        \hbar^L_D:=h\prod_{j\in J}|\sigma_j|^2_{h_j}\Big(\log\big(\vartheta_j|\sigma_j|^2_{h_j}\big)\Big)^{2\alpha},
    \end{align*}
    where $\displaystyle \vartheta_j:=\frac{e^{-2\alpha}}{2\max_X|\sigma_j|^2_{h_j}}<\frac{e^{-2\alpha}}{\max_X|\sigma_j|^2_{h_j}}$, 
    and $\alpha\in\bb{N}$ is a sufficiently large positive integer depending only on the compactness of $X$.
    Then, we obtain 
    \begin{align*}
        1+\frac{2\alpha}{\log\big(\vartheta_j|\sigma_j|^2_{h_j}\big)}>0
    \end{align*}
    on $X$.
    For each $\nu\!\in\!\bb{N}$, we similarly define a singular Hermitian metric $\hbar^L_{D,\nu}$ on $L|_{X\setminus D}$ by 
    \begin{align*}
        \hbar^L_{D,\nu}:=h_\nu\prod_{j\in J}|\sigma_j|^2_{h_j}\Big(\log\big(\vartheta_j|\sigma_j|^2_{h_j}\big)\Big)^{2\alpha}.
    \end{align*}
    Then, by conditions $(i)$ and $(ii)$ of Setup \ref{Setup}, the sequence $\{\hldnu\}_{\nu\in\bb{N}}$ is increases to $\hbar^L_D$. 

    Let $\{\beta\}\in H^q(X,\Omega^n_X(\log D)\otimes L\otimes\scr{I}(h))$ be a cohomology class represented by a smooth form with values in $\Omega^n_X(\log D)\otimes L\otimes\scr{I}(h)$. 
    As in the proof of Theorem \ref{Theorem: Hard Lefschetz if TX>0 on D}, there exists some $\varrho_\nu\in(0,\mu_\nu)$ such that, setting $\hldpn:=\hldnu e^{-\varrho_\nu\psi_\nu}$, we obtain 
    \begin{align*}
        ||\beta||^2_{\nu,\psi,\delta}:=\!\int_{X\setminus(D\cup Z_\nu)}\hspace{-4mm}|\beta|^2_{\wedge^{n,q}\opnd\otimes\hldpn}dV_{\opnd}\leq2||\beta||^2:=2\!\int_X|\beta|^2_{\wedge^{n,q}\omega_P\otimes \hbar^L_D}dV_{\omega_P}<+\infty.
    \end{align*} 
    Note that Lemma \ref{Lemma: log resolusion for (n,q)-forms} also holds for the complex $\Big(\scr{L}^{n,\ast}_{L,\hldpn,\omega_{P,\nu,\delta}},\dbar\Big)$.
    As in the proof of Theorem \ref{Theorem: Hard Lefschetz if TX>0 on D}, the Hilbert space $Z^q\Big(\scr{L}^{n,\ast}_{L,\hldpn,\omega_{P,\nu,\delta}}\Big)$ admits an orthogonal decomposition, and \(\beta\) is a \(\dbar\)-closed form in the Hilbert space induced by the complete \kah metric \(\opnd\) on \(X\setminus (D\cup Z_\nu)\). 
    Therefore, there exists an \(\opnd\)-harmonic form \(u_{\nu,\delta}\) representing the same cohomology class as \(\beta\), satisfying 
    \begin{align*}
        ||u_{\nu,\delta}||^2_{\nu,\psi,\delta}\leq||\beta||^2_{\nu,\psi,\delta}, 
    \end{align*}
    where such a harmonic representative exists only for $\delta > 0$. 

    Let $v_{\nu,\delta}$ be the unique $(n-q,0)$-form such that $u_{\nu,\delta}=v_{\nu,\delta}\wedge\opnd^q$, whose existence follows from the pointwise Lefschetz isomorphism. 
    Then, we obtain
    \begin{align*}
        (q!)^2||v_{\nu,\delta}||^2_{\nu,\psi,\delta}=||u_{\nu,\delta}||^2_{\nu,\psi,\delta}\leq||\beta||^2_{\nu,\psi,\delta}\leq2||\beta||^2,
    \end{align*}
    where \(u_{\nu,\delta}\) is primitive as a \((0,n-q)\)-form. 
    A straightforward computation shows that
    \begin{align*}
        \iO{L,\hldpn}&=\iO{L,h_\nu}+\varrho_\nu\idd\psi_\nu+i\sum^J_{j=1}\Theta_{\Ox{D_j},h_j}\\
        &\qquad+i\sum^J_{j=1}\frac{2\alpha\Theta_{\Ox{D_j},h_j}}{\log\big(\vartheta_j|\sigma_j|^2_{h_j}\big)}+i\sum^J_{j=1}\frac{2\alpha\partial\log|\sigma_j|^2_{h_j}\wedge\dbar\log|\sigma_j|^2_{h_j}}{\Big(\log\big(\vartheta_j|\sigma_j|^2_{h_j}\big)\Big)^2}\\    
        &\geq-2\varepsilon_\nu\omega+i\sum^J_{j=1}\biggl(1+\frac{2\alpha}{\log\big(\vartheta_j|\sigma_j|^2_{h_j}\big)}\biggr)\Theta_{\Ox{D_j},h_j}\\
        &\geq-2\varepsilon_\nu\omega\geq-2\varepsilon_\nu\opnd.
    \end{align*}
    Thus, Proposition \ref{Proposition: Prop in DPS01} can be applied to $\big(X\setminus(D\cup Z_\nu),\opnd,L,\hldpn\big)$, yielding 
    \begin{align*}
        (q!)^2\big|\big|\dbar v_{\nu,\delta}\big|\big|^2_{\nu,\psi,\delta}\leq \frac{2q}{q+1}\varepsilon_\nu||u_{\nu,\delta}||^2_{\nu,\psi,\delta}\leq \frac{4q}{q+1}\varepsilon_\nu||\beta||^2.
    \end{align*}
    These uniform bounds imply that there are subsequences $u_{\nu,\delta_k}$ and $v_{\nu,\delta_k}$ with $\delta_k\to 0$, possessing weak-$L^2$ limits $u_\nu = \lim_{k\to+\infty}u_{\nu,\delta_k}$ and $v_\nu = \lim_{k\to+\infty}v_{\nu,\delta_k}$. 
    Note that, as in the proof of Theorem \ref{Theorem: Hard Lefschetz if TX>0 on D}, there is a stronger/weaker relationship between the $L^2$-topologies on $(0,n-q)$-forms and $(n,q)$-forms.

    The above estimates yield 
    \begin{align*}
        (q!)^2||v_\nu||^2_\nu
        &=||u_\nu||^2_\nu:=\int_{X\setminus D}|u_\nu|^2_{\wedge^{n,q}\omega_P\otimes\hldnu}dV_{\omega_P}\leq||u_\nu||^2_{\nu,\psi,0}\leq2||\beta||^2,\\
        \big|\big|\dbar v_\nu\big|\big|^2_\nu&\leq\frac{4\varepsilon_\nu}{(q+1)!(q-1)!}||\beta||^2,\\
        u_\nu&=\omega^q_P\wedge v_\nu\equiv\beta \qquad \text{in}\quad H^q(X,\Omega^n_X(\log D)\otimes L\otimes\scr{I}(h_\nu e^{-\varrho_\nu\psi_\nu})).
    \end{align*}
    Again, by arguing in a fixed given Hilbert space $L^2(\hbar^L_{D,\nu_0})$, we find $L^2$ convergent subsequences $u_\nu\to u$, $v_\nu\to v$ as $\nu\to+\infty$, and in this way $(q+1)!(q-1)!\big|\big|\dbar v_\nu\big|\big|^2_\nu\leq4\varepsilon_\nu||\beta||^2$ gives $\dbar v=0$, and the above inequalities yield 
    \begin{align*}
        ||v||^2&=\int_{X\setminus D}|v|^2_{\wedge^{0,n-q}\omega_P\otimes\hbar^L_D}dV_{\omega_P}\leq||\beta||^2<+\infty,\\
        u&=\omega^q_P\wedge v\equiv\beta\qquad\text{in}\quad H^q(X,\Omega^n_X(\log D)\otimes L\otimes\scr{I}(h)). 
    \end{align*}
    The last equality can also be verified using the de Rham-Weil isomorphism, where the equisingularity property $\scr{I}(h)=\scr{I}(h_\nu e^{-\varrho_\nu\psi_\nu})$, i.e., condition $(v)$, plays a crucial role. 
    Finally, the condition on the section \(v\), together with Theorem \ref{Theorem: Logarithmic L2-Dolbeault resolusion}, yields 
    \begin{align*}
        v\in \Gamma\big(X\setminus D,\scr{L}^{n-q,0}_{L,\hbar^L_D,\omega_P}\big)\cap\rom{Ker}\,\dbar=H^0(X,\Omega_X^{n-q}(\log D)\otimes L\otimes\scr{I}(h)).
    \end{align*}
    This completes the proof.
\end{proof}

The following corollary holds immediately.

\begin{corollary}\label{Corollary: hard Lefschetz for semi-positive of L and D_j}
    Let $X,D,\omega_P$ be as in the setting of Theorem \ref{Theorem: Hard Lefschetz if TX>0 on D}. 
    If a line bundle $L\longrightarrow X$ and each line bundle $\Ox{D_j}$ are semi-positive, then the wedge multiplication operator $\omega_P^q\wedge\bullet$ induces a surjective morphism 
    \begin{align*}
        \Phi^q_{\omega_P}:H^0(X,\Omega_X^{n-q}(\log D)\otimes L)\longrightarrow H^q(X,\Omega_X^n(\log D)\otimes L)
    \end{align*}
    for every nonnegative integer $q$.
\end{corollary}

\section{A counterexample}

\begin{remark}\label{Remark: counterexample for nef case}
    The surjectivity of \(\Phi^q_{\omega_P}\) in Corollary \ref{Corollary: hard Lefschetz for semi-positive of L and D_j} does not necessarily hold when both \(L\) and \(\Ox{D}\) are nef, as the following counterexample shows.
\end{remark}

\begin{counterexample}\label{Counterexample: counterexample following DPS01}
    Let $B$ be an elliptic curve and $V$ the rank let $2$ vector bundle over $B$ which is defined as the (unique) non split extension 
    \begin{align*}
        0\longrightarrow\cal{O}_B\longrightarrow V\longrightarrow\cal{O}_B\longrightarrow 0.
    \end{align*}
    In particular, the bundle $V$ is numerically flat, i.e., $c_1(V) = 0, c_2(V) = 0$. We consider the ruled surface $X = \bb{P}(V)$. On that surface there is a unique section $C = \bb{P}(\cal{O}_B) \subset X$ with $C^2 = 0$ and $\Ox{C}=\cal{O}_{\bb{P}(V)}(1)$ is a nef line bundle. 
    For any smooth \kah metric $\omega_P$ on $X\setminus C$ which is of \Pt along $C$ and all $k\geq3$, the map 
    \begin{align*}
        \Phi_{\omega_P}:H^0(X,\Omega^1_X(\log C)\otimes\cal{O}_{\bb{P}(V)}(k))\longrightarrow H^1(X,\Omega^2_X(\log C)\otimes\cal{O}_{\bb{P}(V)}(k))\simeq\bb{C}\ne0,
    \end{align*}
    induced by the wedge multiplication operator $\omega_P\wedge\bullet$, is not surjective and, in particular, is the zero map. 
\end{counterexample}

\begin{proof}
It is easy to see that $h^0(X,\Ox{mC})=h^0(X,\cal{O}_{\bb{P}(V)}(m))=h^0(B,S^mV)=1$ for all $m\in\bb{N}$.
%
We claim that 
\begin{align*}
    h^0(X,\Omega^1_X(\log C)(kC))=2
\end{align*}
for all $k\geq2$. 
To prove this, consider the exact sequence
\begin{align*}
    0\longrightarrow\pi^*\Omega^1_C\simeq\cal{O}_X\longrightarrow\Omega^1_X(\log C)\longrightarrow\Omega^1_{X/C}(\log C)\longrightarrow0
\end{align*}
tensored with $\Ox{kC}$, and note that $\Omega^1_{X/C}(\log C)\simeq\Ox{-C}$, which follows from $\Omega^1_{X/C}(\log C)|_F=\Omega^1_{\bb{P}}(\log \{p\})\simeq\cal{O}_{\bb{P}}(-1)$ since \(C\cap F=\{p\}\) is a single point on the fiber $F\simeq\bb{P}$.
From this, we obtain the following long exact sequence
\begin{align*}
    0\longrightarrow H^0(X,\Ox{kC})\longrightarrow H^0(X,\Omega^1_X(\log C)(kC))\longrightarrow H^0(X,\Ox{(k-1)C})\longrightarrow\cdots,
\end{align*}
where $h^0(X,\Ox{kC})=1$ for all $k\geq0$, and hence $h^0(X,\Omega^1_X(\log C)(kC))\leq2$. 
Combining this with the injection $H^0(X,\Omega^1_X(kC))\longrightarrow H^0(X,\Omega^1_X(\log C)(kC))$ 
and the known fact that $h^0(X,\Omega^1_X(kC))=2$ for all $k\geq2$ (see \cite[$\S$2.5]{DPS01}), our claim follows.

Since $K_X\simeq\Ox{-2C}$, setting $m=k-2\geq0$, we have $K_X(kC)\simeq\Ox{mC}$.
By the vanishing $R^q\pi_*\Ox{mC}=0$ for $q>0$ and the Leray spectral sequence, we obtain $H^1(X,\Ox{mC})\simeq H^1(B,S^mV)$, where $\pi_*\Ox{mC}\simeq S^mV$.
Since \(B\) is an elliptic curve, the Riemann-Roch theorem gives $\chi(B,S^mV)=\rom{deg}\,S^mV\!+\rom{rk}(S^mV)(1-g_B)=0$. 
Hence, for all $k\geq2$, we obtain 
\begin{align*}
    h^1(X,K_X(kC))=h^1(B,S^mV)=h^0(B,S^mV)=1.
\end{align*}

We consider the diagram 
\[
\begin{tikzcd}
H^0(X,\Omega_X^1(\log C)(2C))
  \arrow[r, "\omega_P\wedge"]
  \arrow[d, "\cong"']
&
H^1(X,\Omega_X^2(\log C)(2C))=H^1(X,K_X(3C))
  \arrow[d, "\varphi"]
\\
H^0(X,\Omega_X^1(\log C)(3C))
  \arrow[r, "\omega_P\wedge", "\psi"']
&
H^1(X,\Omega_X^2(\log C)(3C))=H^1(X,K_X(4C)).
\end{tikzcd}
\]
It is commutative since $\omega_P\wedge(\sigma_C u)=\sigma_C(\omega_P\wedge u)$, for every $u\in H^0(X,\Omega_X^1(\log C)(2C))$. 
The cohomology sequence of 
\begin{align*}
    0\longrightarrow K_X(3C)\longrightarrow K_X(4C)\simeq\Ox{2C}\longrightarrow K_X(4C)|_C\simeq\cal{O}_C\longrightarrow0
\end{align*}
implies $\varphi=0$. Indeed, for the induced long exact sequence 
\begin{align*}
    \longrightarrow H^0(X,K_X(4C))\overset{r}{\longrightarrow}H^0(C,\cal{O}_C)\overset{\delta}{\longrightarrow}H^1(X,K_X(3C))\overset{\varphi}{\longrightarrow}H^1(X,K_X(4C))\longrightarrow,
\end{align*}
since \(\sigma_C|_C=0\), the restriction map \(r\) is also zero, and hence $\delta$ is injective by exactness. 
Furthermore, $\delta$ is an isomorphism by $h^1(X,K_X(3C))=1$, and hence $\rom{Ker}\,\varphi=H^1(X,K_X(3C))$. 
Therefore, the diagram implies $\psi=0$.
For any \(k\geq2\), the surjectivity of $\psi$ is likewise lost by replacing \((2C,3C)\) with \((kC,(k+1)C)\).
\end{proof}

\vspace*{3mm}

\noindent
{\bf Acknowledgement.} 
The author is supported by Grant-in-Aid for Research Activity Start-up $\sharp$26K16989 from the Japan Society for the Promotion of Science (JSPS).







\end{document}